\documentclass[final]{siamltex}

\usepackage[english]{babel}
\usepackage{latexsym}
\usepackage{amssymb}

\usepackage{amsmath}
\usepackage{relsize}
\usepackage{tikz}
\usepackage{enumerate}
\usepackage{graphicx}
\usepackage{url}
\makeatletter
\newcommand\xleftrightarrow[2][]{\ext@arrow 0099{\longleftrightarrowfill@}{#1}{#2}}
\def\longleftrightarrowfill@{\arrowfill@\leftarrow\relbar\rightarrow}
\makeatother

\newcommand{\nc}{\newcommand}
\nc{\nt}{\newtheorem}
\nt{thm}{Theorem}[section]
\nt{cor}[thm]{Corollary}
\nt{prop}[thm]{Proposition}
\nt{obs}[thm]{Observation}
\nt{claim}[thm]{Claim}
\nt{lem}[thm]{Lemma}
\nt{defn}[thm]{Definition}
\nt{exa}[thm]{Example}
\nt{rem}[thm]{Remark}
\nt{ass}[thm]{Assumption}
\nt{alg}[thm]{Algorithm}
\nt{con}[thm]{Conjecture}
\nc{\ip}[2]{\mbox{$\langle #1,#2 \rangle$}}
\nc{\pf}{\noindent{\bf Proof.\ \ }}
\nc{\finpf}{\hfill{$\Box$}\linespace}
\nc{\linespace}{\vspace{\baselineskip} \noindent}
\nc{\R}{{\bf R}}
\nc{\E}{{\bf E}}
\nc{\cl}{\mbox{\rm cl}\,}
\nc{\cls}{ \mbox{{\scriptsize {\rm cl}}}\,}
\nc{\conv}{\mbox{\rm conv}}
\nc{\rb}{\mbox{\rm rb}\,}
\nc{\ri}{\mbox{\rm ri}\,}
\nc{\inter}{\mbox{\rm int}\,}
\nc{\kernel}{\mbox{\rm ker}\,}
\nc{\bd}{\mbox{\rm bd}\,}
\nc{\spann}{\mbox{\rm span}\,}
\nc{\rint}{\mbox{\rm rint}\,}
\nc{\epi}{\mbox{\rm epi}\,}
\nc{\gph}{\mbox{\rm gph}\,}
\nc{\rge}{\mbox{\rm rge}\,}
\nc{\rgel}{\mbox{\rm {\scriptsize rge}}\,}
\nc{\sepi}{\mbox{\rm {\scriptsize epi}}\,}
\nc{\sbd}{\mbox{\rm {\scriptsize bd}}\,}
\nc{\dom}{\mbox{\rm dom}\,}
\nc{\lin}{\mbox{\rm lin}\,}
\nc{\detr}{\mbox{\rm det}\,}
\nc{\para}{\mbox{\rm par}\,}
\nc{\crit}{\mbox{\rm crit}\,}
\nc{\cone}{\mbox{\rm cone}\,}
\nc{\Diag}{\mbox{\rm Diag}\,}
\nc{\fix}{\mbox{\rm Fix}}
\nc{\aff}{\mbox{\rm aff}\,}
\nc{\tr}{\mbox{\rm tr}\,}

\newcommand{\lf}{\operatornamewithlimits{liminf}}

\newenvironment{myequation}{\begin{equation}}{\end{equation}}

\newenvironment{myeqnarray*}{\begin{eqnarray*}}{\end{eqnarray*}}
\nc{\bmye}{\begin{myequation}} \nc{\emye}{\end{myequation}}

\def\tto{\;{\lower 1pt \hbox{$\rightarrow$}}\kern -12pt
           \hbox{\raise 2.8pt \hbox{$\rightarrow$}}\;}

\title{Orthogonal Invariance and Identifiability}

\author{
A. Daniilidis\thanks{
    Departament de Matem\`{a}tiques, C1/364,
    Universitat Aut\`{o}noma de Barcelona,
    E-08193 Bellaterra, Spain (on leave) and
    DIM-CMM, Universidad de Chile, Blanco Encalada~2120,
    piso~5, Santiago, Chile;
    {\tt
    http://mat.uab.es/{\raise.17ex\hbox{$\scriptstyle\sim$}}arisd}.
    Research supported by the grant MTM2011-29064-C01 (Spain) and FONDECYT Regular No 1130176 (Chile).
    }
    \and
D. Drusvyatskiy\thanks{%
    Department of Operations Research and Information Engineering,
    Cornell University,
    Ithaca, New York, USA;
    {\tt http://people.orie.cornell.edu/dd379/}.
    Work of Dmitriy Drusvyatskiy on this paper has been partially supported by the NDSEG grant from the Department of Defense.
    }%
    \and
  A.S. Lewis\thanks{%
  School of Operations Research and Information Engineering,
  Cornell University,
  Ithaca, New York, USA;
  {\tt http://people.orie.cornell.edu/aslewis/}.
  Research supported in part by National Science Foundation Grant DMS-0806057
  and by the US-Israel Binational Scientific Foundation Grant 2008261.
}}

\begin{document}

\maketitle

\begin{abstract}
Orthogonally invariant functions of symmetric matrices often inherit properties from their diagonal restrictions:  von Neumann's theorem on matrix norms is an early example.  We discuss the example of ``identifiability'', a common property of nonsmooth functions associated with the existence of a smooth manifold of approximate critical points.  Identifiability (or its synonym, ``partial smoothness'') is the key idea underlying active set methods in optimization.  Polyhedral functions, in particular, are always partly smooth, and hence so are many standard examples from eigenvalue optimization.
\end{abstract}



\pagestyle{myheadings}
\thispagestyle{plain}
\markboth{A. Daniilidis, D. Drusvyatskiy, A.S. Lewis}{Orthogonal Invariance and Identifiability}

\section{Introduction}
Nonsmoothness is inherently present throughout even classical mathematics and engineering - the spectrum of a symmetric matrix variable is a good example.  The nonsmooth behavior is not, however, typically pathological, but on the contrary is highly structured. The theory of {\em identifiability} (or its synonym, {\em partial smoothness})
\cite{Lewis-active,Hare,Wright,ident} models this idea by positing existence of smooth manifolds capturing
the full ``activity'' of the problem. Such manifolds, when they exist, are simply composed of approximate critical points of the minimized function. In the classical case of nonlinear programming, this theory reduces
to the active set philosophy. Illustrating the ubiquity of the notion, the authors of \cite{gen} prove that
identifiable manifolds exist generically for convex semi-algebraic optimization problems.

Identifiable manifolds are particularly prevalent in the
context of eigenvalue optimization. One of our goals is to shed new
light on this phenomenon. To this end, we will consider so-called
{\em spectral functions}. These are functions $F$, defined on the
space of symmetric matrices ${\bf S}^n$, that depend on matrices
only through their eigenvalues, that is, functions that are
invariant under the action of the orthogonal group by conjugation.
Spectral functions can always be written as the composition
$F=f\circ \lambda$ where $f$ is a permutation-invariant function on
$\R^n$ and $\lambda$ is the mapping assigning to each matrix
$X\in{\bf S}$ the vector of its eigenvalues
$(\lambda_1(X),\ldots,\lambda_n(X))$ in non-increasing order,
see~\cite[Section~5.2]{cov_lift}. Notable examples of functions
fitting in this category are $X\mapsto\lambda_1(X)$ and $X\mapsto
\sum^{n}_{i=1} |\lambda_i(X)|$. Though the spectral mapping
$\lambda$ is very badly behaved, as far as say differentiability is
concerned, the symmetry of $f$ makes up for the fact, allowing
powerful analytic results to become available.

In particular, the {\em Transfer Principle} asserts that
$F$ inherits many geometric (more generally variational analytic) properties of $f$, or equivalently, $F$ inherits many properties of its restriction to diagonal matrices. For example, when $f$ is a permutation-invariant norm, then $F$ is an orthogonally invariant norm on the space of symmetric matrices --- a special case of von Neumann's theorem on unitarily invariant matrix norms \cite{von_Neumann}.  The collection of
properties known to satisfy this principle is impressive: convexity
\cite{lag, cov_orig}, prox-regularity \cite{spec_prox},
Clarke-regularity \cite{send,lag},
smoothness~\cite{lag,der,high_order,man, diff_1,diff_2}, algebraicity~\cite{man},
and stratifiability \cite[Theorem 4.8]{mather}. In this work, we add
identifiability (and partial smoothness) to the list
(Theorems~\ref{prop:lift_id} and \ref{thm:lift_pman}). In
particular, many common spectral functions (like the two
examples above) can be written in the composite form $f\circ
\lambda$, where $f$ is a permutation-invariant convex {\em polyhedral} function.
As a direct corollary of our results, we conclude that such
functions always admit partly smooth structure! Furthermore, a
``polyhedral-like'' duality theory of partly smooth manifolds
becomes available.

One of our intermediary theorems is of particular interest. We will
give an elementary argument showing that a permutation-invariant set $M$ is a
${\bf C}^{\infty}$ manifold if and only if the spectral set
$\lambda^{-1}(M)$ is a ${\bf C}^{\infty}$ manifold
(Theorem~\ref{thm:lift_man}). The converse implication of our result
is apparently new. On the other hand, the authors of \cite{man}
proved the forward implication even for ${\bf C}^k$ manifolds (for
$k=2,\ldots,\infty$). This being said, their proof is rather long and dense, whereas the proof of our result is very
accessible. The key idea of our approach is to consider the metric projection onto $M$.

The outline of the manuscript is as follows. In
Section~\ref{sec:lift_man} we establish some basic notation and give
an elementary proof of the spectral lifting property for ${\bf
C}^{\infty}$ manifolds. In Section~\ref{sec:lift_id} we prove the
lifting property for identifiable sets and partly smooth manifolds,
while in Section~\ref{sec:dual} we explore duality theory of partly
smooth manifolds. Section~\ref{sec:nonsymm} illustrates how our
results have natural analogues in the world of nonsymmetric
matrices.

\section{Spectral functions and lifts of manifolds}\label{sec:lift_man}
\subsection{Notation}
Throughout, the symbol ${\bf E}$ will denote a Euclidean space (by
which we mean a finite-dimensional real inner-product space). The
functions that we will be considering will take their values in the
extended real line $\overline{\R}:=\R\cup\{-\infty,\infty\}$. We say
that an extended-real-valued function is {\em proper} if it is never
$\{-\infty\}$ and is not always $\{+\infty\}$. For a set $Q\subset
{\bf E}$, the {\em indicator function} $\delta_Q\colon {\bf
E}\to\overline{\R}$ is a function that takes the value $0$ on $Q$
and $+\infty$ outside of $Q$. An open ball of radius $\epsilon$
around a point $\bar{x}$ will be denoted by $B_{\epsilon}(\bar{x})$,
while the open unit ball will be denoted by ${\bf B}$. Two
particular realizations of ${\bf E}$ will be important for us,
namely $\R^n$ and the space ${\bf S}^n$ of $n\times n$-symmetric
matrices.

Throughout, we will fix an orthogonal basis of $\R^n$,
along with an inner product $\langle \cdot,\cdot\rangle$. The
corresponding norm will be written as $\|\cdot\|$. The group of
permutations of coordinates of $\R^n$ will be denoted by $\Sigma^n$,
while an application of a permutation $\sigma\in\Sigma^n$ to a point
$x\in\R^n$ will simply be written as $\sigma x$. We denote by
$\R^n_{\geq}$ the set of all points $x\in\R^n$ with $x_1\geq
x_2\geq\ldots \geq x_n$. A function $f\colon\R^n\to\overline{\R}$ is
said to be {\em symmetric} if we have $f(x)=f(\sigma x)$ for every
$x\in\R^n$ and every $\sigma\in\Sigma^n$.

The vector space of real $n\times n$ symmetric matrices ${\bf S}^n$
will always be endowed with the trace inner product $\langle
X,Y\rangle=\tr(XY)$, while the associated norm (Frobenius norm) will
be denoted by $\|\cdot\|_{F}$. The group of orthogonal $n\times n$
matrices will be denoted by ${\bf O}^n$. Note that the group of
permutations $\Sigma^n$ naturally embeds in ${\bf O}^n$. The action
of ${\bf O}^n$ by conjugation on ${\bf S}^n$ will be written as
$U.X:=U^{T}XU$, for matrices $U\in {\bf O}^n$ and $X\in {\bf S}^n$.
A function $h\colon{\bf S}^n\to\overline{\R}$ is said to be {\em
spectral} if we have $h(X)=h(U.X)$ for every $X\in{\bf S}^n$ and
every $U\in{\bf O}^n$.

\subsection{Spectral functions and the transfer principle}
We can now consider the spectral mapping $\lambda\colon{\bf
S}^n\to\R^n$ which simply maps symmetric matrices to the vector of
its eigenvalues in nonincreasing order. Then a function on ${\bf
S}^n$ is spectral if and only if it can be written as a composition
$f\circ\lambda$, for some symmetric function
$f\colon\R^n\to\overline{\R}$. (See for example
\cite[Proposition~4]{lag}.) As was mentioned in the introduction,
the {\em Transfer Principle} asserts that a number of
variational-analytic properties hold for the spectral function
$f\circ\lambda$ if and only if they hold for $f$. We will encounter
a number of such properties in the current work. Evidently,
analogous results hold even when $f$ is only {\em locally symmetric}
(to be defined below). The proofs follow by a reduction to the
symmetric case by simple symmetrization arguments, and hence we will
omit details in the current paper.

For each point $x\in\R^n$, we consider the {\em stabilizer}
$$\fix(x):=\{\sigma\in\Sigma^n : \sigma x=x\}.$$
\begin{defn}[Local symmetry]
{\rm
A function $f\colon\R^n\to\overline{\R}$ is {\em locally symmetric}
at a point $\bar{x}\in\R^n$ if we have $f(x)=f(\sigma x)$ for all
points $x$ near $\bar{x}$ and all permutations $\sigma\in \fix(\bar{x})$.
}
\end{defn}

\smallskip
A set $Q\subset\R^n$ is symmetric (respectively locally symmetric) if the indicator
function $\delta_Q$ is symmetric (respectively locally symmetric). The following shows
that smoothness satisfies the Transfer Principle \cite{diff_2,diff_1}.

\smallskip
\begin{thm}[Lifts of smoothness]\label{thm:sm_lift}
Consider a function $f\colon\R^n\to\overline{\R}$ and a matrix
$\overline{X}\in{\bf S}^n$. Suppose that $f$ is locally symmetric
around $\bar{x}:=\lambda(\overline{X})$. Then $f$ is ${\bf
C}^p$-smooth $(p=1,\ldots,\infty )$ around $\bar{x}$ if and only if
the spectral function $f\circ\lambda$ is ${\bf C}^p$-smooth around
$\overline{X}$.
\end{thm}

\smallskip
The {\em distance} of a point $\bar{x}$ to a set $Q\subset {\bf E}$ is simply
$$d_Q(x):=\inf\,\{\|x-y\|: y\in Q\},$$
whereas the {\em metric projection} of $x$ onto $Q$ is defined by
$$P_Q(x):=\{y\in Q: d_Q(x)=y\}.$$
It will be important for us to relate properties of a set $Q$ with those of the
metric projection $P_Q$. To this end, the following notion arises naturally \cite{prox_reg,loc_diff}.
\smallskip
\begin{defn}[Prox-regularity]\label{defn:prox_set}
{\rm A set $Q\subset {\bf E}$ is {\em prox-regular} at a point
$\bar{x}\in Q$ if $Q$ is locally closed around $\bar{x}$ and the
projection mapping $P_Q$ is single-valued around~$\bar{x}$. }
\end{defn}

\smallskip
In particular, all ${\bf C}^2$-manifolds and all closed convex sets
are prox-regular around any of their points. See for example
{\cite[Example 13.30, Proposition 13.32]{VA}}. Additionally, it is
well-known that if $M\subset \E$ is a ${\bf C}^p$ smooth manifold
(for $p\geq 2$) around a point $\bar{x}\in M$, then there exists a
neighborhood $U$ of $\bar{x}$ on which the projection $P_M$ is
single-valued and ${\bf C}^{p-1}$-smooth. Prox-regularity also
satisfies the transfer principle \cite[Proposition 2.3, Theorem
2.4]{spec_prox}.

\smallskip
\begin{thm}[Lifts of prox-regularity]\label{thm:lift_spec}
Consider a matrix $\overline{X}\in{\bf S}^n$ and a set $Q\subset\R^n$ that is locally symmetric around the point $\bar{x}:=\lambda(\overline{X})$. Then the function $d_Q$ is locally symmetric near $\bar{x}$ and the distance to the spectral set $\lambda^{-1}(Q)$ satisfies
$$d_{\lambda^{-1}(Q)}=d_Q\circ\lambda, \textrm{ locally around } \overline{X}.$$
Furthermore, $Q$ is prox-regular at $\bar{x}$ if and only if $\lambda^{-1}(Q)$ is prox-regular at $\overline{X}$.
\end{thm}
\smallskip

If a set $Q\subset {\bf E}$ is prox-regular at $\bar{x}$, then the {\em proximal normal cone}
$$N_Q(\bar{x}):=\R_{+}\{v\in {\bf E}: \bar{x}\in P_Q(\bar{x}+v)\},$$
and the {\em tangent cone}
$$T_Q(\bar{x}):=\Big\{\lim_{i\to\infty} \lambda_i (x_i-\bar{x}): \lambda_i\uparrow \infty \textrm{ and } x_i\in Q\Big\}.$$
are closed convex cones and are polar to each other \cite[Corollary 6.29]{VA}. Here, we mean polarity
in the standard sense of convex analysis, namely for any closed convex cone $K\subset {\bf E}$, the polar
of $K$ is another closed convex cone defined by $$K^{o}:=\{v\in {\bf E}:\langle v,w \rangle\leq 0 \textrm{ for all } w\in K\}.$$

\subsection{Lifts of symmetric manifolds}
It turns out (not surprisingly) that smoothness of the projection
$P_Q$ is inherently tied to smoothness of $Q$ itself, which is the
content of the following lemma.

For any mapping $F\colon {\bf E}\to {\bf E}$, the directional derivative of $F$ at $\bar{x}$ in direction $w$ (if it exists) will be denoted by
$$DF(\bar{x})(w):=\lim_{t\downarrow 0}\frac{F(\bar{x}+tw)-F(\bar{x})}{t},$$ while the G\^{a}teaux derivative of $F$ at $\bar{x}$ (if it exists) will be denoted by $DF(\bar{x})$.

\smallskip
\begin{lem}[Smoothness of the metric projection]\label{lem:smooth_proj}
Consider a set $Q\subset{\bf E}$ that is prox-regular at a point $\bar{x}\in Q$. Then
\begin{equation}\label{eqn:ker}
DP_Q(\bar{x})(v)=0, \quad\textrm{ for any } v\in N_Q(\bar{x}).
\end{equation}
If $P_Q$ is directionally differentiable at $\bar{x}$, then we also have
\begin{equation}\label{eqn:tan}
DP_Q(\bar{x})(w)=w, \quad\textrm{ for any } w\in T_Q(\bar{x}).
\end{equation}
In particular, if $P_Q$ is G\^{a}teaux differentiable at $\bar{x}$, then $N_Q(\bar{x})$ and $T_Q(\bar{x})$ are orthogonal subspaces and $DP_Q(\bar{x})=P_{T_Q(\bar{x})}$. If $P_Q$ is ${\bf C}^k$ ($k=1,\ldots,\infty$) smooth near $\bar{x}$, then $P_Q$ automatically has constant rank near $\bar{x}$ and consequently $Q$ is a ${\bf C}^k$ manifold around $\bar{x}$.
\end{lem}
\begin{proof}
Observe that for any normal vector $\bar{v}\in N_Q(\bar{x})$ there exists $\epsilon >0$ so that $P_Q(\bar{x}+\epsilon'\bar{v})=\bar{x}$ for all nonnegative $\epsilon' <\epsilon$. Equation (\ref{eqn:ker}) is now immediate.

Suppose now that $P_Q$ is directionally differentiable at $\bar{x}$ and consider a vector $w\in T_Q(\bar{x})$ with $\|w\|=1$. Then there exists a sequence $x_i\in Q$ converging to $\bar{x}$ and satisfying
$w=\lim_{i\to\infty} \frac{x_i-\bar{x}}{\|x_i-\bar{x}\|}$. Define $t_i:=\|x_i-\bar{x}\|$
and observe that since $P_Q$ is Lipschitz continuous, for some constant $L$ we have
$$\frac{\|P_Q(\bar{x}+t_i w)-P_Q(x_i)\|}{t_i}\leq L\Big\|w-\frac{x_i-\bar{x}}{t_i}\Big\|,$$ and consequently this quantity converges to zero. We obtain
$$DP_Q(\bar{x})(w)=\lim_{i\to\infty}\frac{P_Q(\bar{x}+t_i w)-\bar{x}}{t_i}=\lim_{i\to\infty}\frac{P_Q(x_i)-\bar{x}}{t_i}=w,$$
as claimed.

Suppose now that $P_Q$ is G\^{a}teaux differentiable at $\bar{x}$. Then clearly from (\ref{eqn:ker}) we have $N_Q(\bar{x})\subset \ker DP_Q(\bar{x})$. If $N_Q(\bar{x})$ were a proper convex subset of $\ker DP_Q(\bar{x})$, then we would deduce $$T_Q(\bar{x})\cap  \ker DP_Q(\bar{x}) =[N_Q(\bar{x})]^{\circ}\cap \ker DP_Q(\bar{x})\neq \{0\},$$ thereby contradicting equation (\ref{eqn:tan}). Hence $N_Q(\bar{x})$ and $T_Q(\bar{x})$ are orthogonal subspaces and the equation $DP_Q(\bar{x})=P_{T_Q(\bar{x})}$ readily follows from (\ref{eqn:ker}) and (\ref{eqn:tan}).

Suppose now that $P_Q$ is ${\bf C}^k$-smooth (for $k=1,\ldots,\infty$) around $\bar{x}$.
Then clearly we have
$$\rank DP_Q(x)\geq \rank DP_Q(\bar{x}), \textrm{ for all } x \textrm{ near } \bar{x}.$$
Towards establishing equality above, we now claim that the set-valued mapping $T_Q$ is outer-semicontinuous at $\bar{x}$. To see this, consider sequences $x_i\to\bar{x}$ and $w_i\in T_Q(x_i)$, with $w_i$ converging to some vector $\bar{w}\in {\bf E}$. From equation (\ref{eqn:tan}), we deduce
$w_i=DP_Q(x_i)(w_i)$. Passing to the limit, while taking into account the continuity of $DP_Q$, we obtain $\bar{w}=DP_Q(\bar{x})(\bar{w})$. On the other hand, since $DP_Q(\bar{x})$ is simply the linear projection onto $T_{Q}(\bar{x})$, we deduce the inclusion $\bar{w}\in T_Q(\bar{x})$, and thereby establishing outer-semicontinuity of $T_Q$ at $\bar{x}$. It immediately follows that the inequality,
$\dim T_Q(x) \leq \dim T_Q(\bar{x})$, holds for all $x\in Q$ near $\bar{x}$. 

One can easily verify that for any point $x$ near $\bar{x}$, the inclusion $N_Q(P_Q(x))\subset\ker DP_Q(x)$ holds.
Consequently we deduce
$$\rank DP_Q(x)\leq \dim T_Q(P_Q(x))\leq \dim T_Q(\bar{x})=\rank DP_Q(\bar{x}),$$
for all $x\in {\bf E}$ sufficiently close to $\bar{x}$.
as claimed. Hence $P_Q$ has constant rank near $\bar{x}$. By the constant rank theorem, for all sufficiently small $\epsilon >0$, the set $P_Q(B_{\epsilon}(\bar{x}))$ is a ${\bf C}^k$ manifold. Observing that the set $P_Q(B_{\epsilon}(\bar{x}))$ coincides with $Q$ near $\bar{x}$ completes the proof.
\end{proof}
\smallskip

The following observation will be key. It shows that the metric projection map onto a prox-regular set is itself a gradient of a ${\bf C}^1$-smooth function. This easily follows from \cite[Proposition 3.1]{loc_diff}. In the convex case, this observation has been recorded and used explicitly for example in \cite[Proposition 2.2]{phelps} and \cite[Preliminaries]{holmes}, and even earlier in \cite{asp} and \cite{zar}.

\smallskip
\begin{lem}[Projection as a derivative]\label{lem:proj_dev}
Consider a set $Q\subset {\bf E}$ that is prox-regular at $\bar{x}$. Then the function
$$h(x):=\frac{1}{2}\|x\|^2-\frac{1}{2}d^2_Q(x),$$
is ${\bf C}^1$-smooth on a neighborhood of $\bar{x}$, with
$\nabla h(x)= P_Q(x)$ for all $x$ near $\bar{x}$.
\end{lem}
\smallskip

We are now ready to state and prove the main result of this section.
\smallskip
\begin{thm}[Spectral lifts of manifolds]\label{thm:lift_man}
Consider a matrix $\overline{X}\in{\bf S}^n$ and a set $M\subset\R^n$
that is locally symmetric around $\bar{x}:=\lambda(\overline{X})$.
Then $M$ is a ${\bf C}^{\infty}$ manifold around $\bar{x}$ if and only
if the spectral set $\lambda^{-1}(M)$ is a ${\bf C}^{\infty}$ manifold around $\overline{X}$.
\end{thm}
\begin{proof} Consider the function $$h(x):=
\frac{1}{2}\|x\|^2-\frac{1}{2}d^2_M(x).$$ Suppose that $M$ is a
${\bf C}^{\infty}$ manifold around $\bar{x}$. In particular $M$ is
prox-regular, see \cite[Example 13.30]{VA}. Then using
Theorem~\ref{thm:lift_spec} we deduce that $h$ is locally symmetric
around $\bar{x}$. In turn, Lemma~\ref{lem:proj_dev} implies the
equality $\nabla h=P_M$ near $\bar{x}$. Since $M$ is a ${\bf
C}^{\infty}$ manifold, the projection mapping $P_M$ is ${\bf
C}^{\infty}$-smooth near $\bar{x}$. Combining this with
Theorem~\ref{thm:sm_lift}, we deduce that the spectral function
$h\circ\lambda$ is ${\bf C}^{\infty}$-smooth near $\overline{X}$.
Observe
\begin{align*}
(h\circ\lambda)(X)&=\frac{1}{2}\|\lambda(X)\|^2-\frac{1}{2}d^2_{M}(\lambda(X))\\
&=\frac{1}{2}\|X\|^{2}_{F}-\frac{1}{2}d^2_{\lambda^{-1}(M)}(X),
\end{align*}
where the latter equality follows from Theorem~\ref{thm:lift_spec}.
Applying Theorem~\ref{thm:lift_spec}, we deduce that
$\lambda^{-1}(M)$ is prox-regular at $\overline{X}$. Combining this
with Lemma~\ref{lem:proj_dev}, we obtain equality $\nabla
(h\circ\lambda)(X)=P_{\lambda^{-1}(M)}(X)$ for all $X$ near
$\overline{X}$. Consequently the mapping $X\mapsto
P_{\lambda^{-1}(M)}(X)$ is ${\bf C}^{\infty}$-smooth near
$\overline{X}$. Appealing to Lemma~\ref{lem:smooth_proj}, we
conclude that $\lambda^{-1}(M)$ is a ${\bf C}^{\infty}$ manifold.
The proof of the converse implication is analogous. \end{proof}
\smallskip

\begin{rem}{\rm
The proof of Theorem~\ref{thm:lift_man} falls short of establishing
the lifting pro\-per\-ty for ${\bf C}^k$ manifolds, with $k$ is
finite, but not by much. The reason for that is that ${\bf C}^k$
manifolds yield projections that are only ${\bf C}^{k-1}$ smooth.
Nevertheless, the same proof shows that ${\bf C}^k$ manifolds do
lift to ${\bf C}^{k-1}$ manifolds, and conversely ${\bf C}^k$
manifolds project down by $\lambda$ to ${\bf C}^{k-1}$ manifolds. }
\end{rem}

\subsection{Dimension of the lifted manifold}

The proof of Theorem~\ref{thm:lift_man} is relatively simple and
short, unlike the involved proof of \cite{man}. One shortcoming however is that it does not a priori yield information about the dimension of the lifted manifold $\lambda^{-1}(M)$. In this section, we outline how we can use the fact that $\lambda^{-1}(M)$ is a manifold to establish a formula between the dimensions of $M$ and $\lambda^{-1}(M)$.
This section can safely be skipped upon first reading.

We adhere closely to the notation and some of the combinatorial
arguments of \cite{man} and \cite{manN}. With any point $x\in\R^{n}$
we associate a partition $\mathcal{P}_{x}=\{I_{1},\ldots,I_{\rho}\}$
of the set $\{1,\ldots,n\}$, whose elements are defined as follows:
$$i,j\in I_{\ell}\Longleftrightarrow x_{i}=x_{j}.$$

\noindent It follows readily that for $x\in\R_{\geq}^{n}$ there
exists a sequence
$$
1=i_{0}\leq i_{1}<\ldots<i_{\rho}=n
$$
such that
$$
I_{\ell}=\{i_{\ell-1},\ldots,i_{\ell}\},\text{\quad for each }\ell
\in\{1,\ldots,\rho\}.
$$
For any such partition $\mathcal{P}$ we set
\[
\Delta_{\mathcal{P}}:=\{x\in\R_{\geq}^{n}:\mathcal{P}_{x}=\mathcal{P}\}.
\]
As explained in \cite[Section~2.2]{man}, the set of all such
$\Delta_{\mathcal{P}}$'s defines an affine stratification of
$\R_{\geq}^{n}$. Observe further that for every point
$x\in\R_{\geq}^{n}$ we have
$$\lambda^{-1}(x)=\{U^{T}XU:U\in\mathbf{O}^{n}\}.$$  Let
$\mathbf{O}_{X}^{n}:=\{U\in\mathbf{O}^{n}:U^{T}XU=X\}$ denote the
stabilizer of $X$, which is a ${\bf C}^{\infty}$ manifold of dimension
$$
\dim\mathbf{O}_{X}^{n}\,=\,\dim\left( \prod_{1\leq\ell\leq\rho}
\mathbf{O}^{|I_{\ell}|}\right)  = \sum_{1\leq\ell\leq\rho}
\,\frac{|I_{\ell}|\,(|I_{\ell}|-1)}{2},
$$
as one can easily check. Since the orbit $\lambda^{-1}(x)$ is isomorphic to
$\mathbf{O}^{n}/\mathbf{O}_{X}^{n}$, it follows that it is a
submanifold of $\mathbf{S}^{n}$. A computation, which can be found in \cite{man}, then yields the equation
$$
\dim\,{\lambda^{-1}(x)=\,}\dim\mathbf{O}^{n}-\dim\mathbf{O}_{X}^{n}=
\sum_{1\leq i<j\leq\rho} \,|I_{i}|\,|I_{j}|.
$$

Consider now any locally symmetric manifold $M$ of
dimension~$d$. There is no loss of generality to assume that $M$ is
connected and has nonempty intersection with $\R_{\geq}^{n}$. Let us
further denote by $\Delta_{\ast}$ an affine stratum of the
aforementioned stratification of $\R_{\geq}^{n}$ with the property
that its dimension is maximal among all of the strata $\Delta$
enjoying a nonempty intersection with $M$. It follows that there
exists a point $\bar{x}\in M\cap\Delta_{\ast}$ and $\delta>0$ satisfying
$M\cap B(\bar{x},\delta)\subset\Delta_{\ast}$ (see
\cite[Section~3]{man} for details). Since
$\dim\,{\lambda^{-1}(M)=\dim \,\lambda^{-1}(M\cap
B(\bar{x},\delta))}$ and since $\lambda^{-1}(M\cap
B(\bar{x},\delta))$ is a fibration we obtain
\begin{equation}
\dim\,\lambda^{-1}(M)\,=\,\dim\,{M\,}+ \sum_{1\leq
i<j\leq\rho_{\ast}}
\,|I_{i}^{\ast}|\,\,|I_{j}^{\ast}|,\label{dimension}
\end{equation}
where $\mathcal{P}_{\ast}=\{I_{1}^{\ast},\ldots,I_{\rho}^{\ast}\}$
is the partition associated to $\bar{x}$ (or equivalently, to any
$x\in\Delta_{\ast}$).

\smallskip
\begin{rem} {\rm It's worth to point out that it is possible to have strata
$\Delta_{1} \neq \Delta_{2}$ of $\R_{\geq}^{n}$ of the same
dimension, but giving rise to stabilizers of different dimension for
their elements. The argument above shows that a connected locally symmetric manifold cannot
intersect simultaneously these strata. This also follows implicitly
from the forthcoming Lemma~\ref{nhad}, asserting the connectedness
of ${\lambda^{-1}(M)}$, whenever $M$ is connected.}
\end{rem}

\section{Spectral lifts of identifiable sets and partly smooth manifolds}\label{sec:lift_id}
We begin this section by summarizing some of the basic tools used in
variational analysis and nonsmooth optimization. We refer the reader
to the monographs of Borwein-Zhu \cite{Borwein-Zhu},
Clarke-Ledyaev-Stern-Wolenski \cite{CLSW}, Mordukhovich
\cite{Mord_1}, and Rockafellar-Wets \cite{VA} for more details. Unless otherwise
stated, we follow the terminology and notation of \cite{VA}.

\subsection{Variational analysis of spectral functions}
For a function $f\colon {\bf E}\rightarrow\overline{\R}$, the {\em
domain} of $f$ is $$\mbox{\rm dom}\, f:=\{x\in {\bf E}:
f(x)<+\infty\},$$ and the {\em epigraph} of $f$ is $$\mbox{\rm
epi}\, f:= \{(x,r)\in {\bf E}\times\R: r\geq f(x)\}.$$ We will say
that $f$ is {\em lower semicontinuous} ({\em lsc} for short) at a
point $\bar{x}$ provided that the inequality
$\lf_{x\to\bar{x}}f(x)\geq f(\bar{x})$ holds. If $f$ is lower
semicontinuous at every point, then we will simply say that $f$ is
lower semicontinuous. For any set $Q$, the symbols $\mbox{\rm cl}\,
Q$, $\mbox{\rm conv}\, Q$, and $\aff Q$ will denote the topological
closure, the convex hull, and the affine span of $Q$ respectively.
The symbol $\para Q$ will denote the parallel subspace of $Q$,
namely the set $\para Q:=(\aff Q)-Q$. For convex sets $Q\in\E$, the
symbols $\ri Q$ and $\rb Q$ will denote the relative interior and
the relative boundary of $Q$, respectively.

Given any set $Q\subset\E$ and a mapping $f\colon Q\to
\widetilde{Q}$, where $\widetilde{Q}$ is a subset of some other
Euclidean space ${\bf F}$, we say that $f$ is ${\bf C}^p$-{\em
smooth} if for each point $\bar{x}\in Q$, there is a neighborhood
$U$ of $\bar{x}$ and a ${\bf C}^p$-smooth mapping $\widehat{f}\colon
\E\to{\bf F}$ that agrees with $f$ on $Q\cap U$.

Recall that by Theorem~\ref{thm:sm_lift}, smoothness of functions
satisfies the Transfer Principle. Shortly, we will need a slightly
strengthened version of this result, where smoothness is considered
only relative to a certain locally symmetric subset. We record it
now.
\smallskip
\begin{cor}[Lifts of restricted smoothness]\label{cor:rest_smooth}
Consider a function $f\colon\R^n\to\overline{\R}$, a matrix $\overline{X}\in{\bf S}^n$, and a set $M\subset \R^n$ containing $\bar{x}:=\lambda(\bar{X})$. Suppose that $f$ and $M$ are locally symmetric around $\bar{x}$. Then the restriction of $f$ to $M$ is ${\bf C}^p$-smooth ($p=1,\ldots,\infty$) around $\bar{x}$ if and only if the restriction of $f\circ\lambda$ to $\lambda^{-1}(M)$ is ${\bf C}^p$-smooth around $\overline{X}$.
\end{cor}
\begin{proof}
Suppose that the restriction of $f$ to $M$ is ${\bf C}^p$-smooth around $\bar{x}$. Then there exists a ${\bf C}^p$-smooth function $\tilde{f}$, defined on $\R^n$, and agreeing with $f$ on $M$ near $\bar{x}$. Consider then the symmetrized function
$$\tilde{f}_{{\scriptsize {\rm sym}}}(x):=\frac{1}{|\fix(\bar{x})|}\sum_{\sigma\in\fix(\bar{x})}\tilde{f}(\sigma x),$$
where $|\fix(\bar{x})|$ denotes the cardinality of the set $\fix(\bar{x})$. Clearly $\tilde{f}_{{\scriptsize {\rm sym}}}$ is ${\bf C}^p$-smooth, locally symmetric around $\bar{x}$, and moreover it agrees with $f$ on $M$ near $\bar{x}$. Finally, using Theorem~\ref{thm:sm_lift}, we deduce that the spectral function $\tilde{f}_{{\scriptsize {\rm sym}}}\circ\lambda$ is ${\bf C}^p$-smooth around $\overline{X}$ and it agrees with $f\circ{\lambda}$ on $\lambda^{-1}(M)$ near $\overline{X}$. This proves the forward implication of the corollary.

\noindent To see the converse, define $F:=f\circ\lambda$, and
suppose that the restriction of $F$ to $\lambda^{-1}(M)$ is ${\bf
C}^p$-smooth around $\overline{X}$. Then there exists a ${\bf
C}^p$-smooth function $\widetilde{F}$, defined on ${\bf S}^n$, and
agreeing with $F$ on $\lambda^{-1}(M)$ near $\overline{X}$. Consider
then the function
$$\widetilde{F}_{{\scriptsize {\rm sym}}}(X):=\frac{1}{|{\bf O}^n|}\sum_{U\in{\bf O}^n}\widetilde{F}(U.X),$$
where $|{\bf O}^n|$ denotes the cardinality of the set ${\bf O}^n$. Clearly $\widetilde{F}_{{\scriptsize {\rm sym}}}$ is ${\bf C}^p$-smooth, spectral, and it agrees with $F$ on $\lambda^{-1}(M)$ near $\overline{X}$. Since $\widetilde{F}_{{\scriptsize {\rm sym}}}$ is spectral, we deduce that there is a symmetric function $\tilde{f}$ on $\R^n$ satisfying $\widetilde{F}_{{\scriptsize {\rm sym}}}=\tilde{f}\circ\lambda$. Theorem~\ref{thm:sm_lift} then implies that $\tilde{f}$ is ${\bf C}^p$-smooth. Hence to complete the proof, all we have to do is verify that $\tilde{f}$ agrees with $f$ on $M$ near $\bar{x}$. To this end consider a point $x\in M$ near $\bar{x}$ and choose a permutation $\sigma\in\fix(\bar{x})$ satisfying $\sigma x\in\R^n_{\geq}$. Let $U\in {\bf O}^n$ be such that $\overline{X}=U^T(\Diag \bar{x})U$. Then we have
$$\tilde{f}(x)=\tilde{f}(\sigma x)=\widetilde{F}_{{\scriptsize {\rm sym}}}\big(U^T(\Diag x)U\big)=F\big(U^T(\Diag x)U\big)=f(\sigma x)=f(x),$$ as claimed.
\end{proof}
\smallskip

Subdifferentials are the primary variation-analytic tools for studying general nonsmooth functions $f$ on $\E$.
\smallskip
\begin{defn}[Subdifferentials]
{\rm Consider a function $f\colon\E\to\overline{\R}$ and a point $\bar{x}$ with $f(\bar{x})$ finite.
\begin{itemize}
\item[{\rm (i)}] The {\em Fr\'{e}chet subdifferential} of $f$ at $\bar{x}$, denoted
$\hat{\partial}f(\bar{x})$, consists of all vectors $v \in \E$ satisfying $$f(x)\geq f(\bar{x})+\langle v,x-\bar{x} \rangle +o(\|x-\bar{x}\|).$$
\item [{\rm (ii)}] The {\em limiting subdifferential} of $f$ at $\bar{x}$, denoted $\partial f(\bar{x})$, consists of all vectors $v\in\E$ for which there exist sequences $x_i\in\E$ and $v_i\in\hat{\partial} f(x_i)$ with $(x_i,f(x_i),v_i)\to(\bar{x},f(\bar{x}),v)$.
\end{itemize}}
\end{defn}
\smallskip


Let us now recall from \cite[Proposition 2]{lag} the following
lemma, which shows that subdifferentials behave as one would expect
in presence of symmetry.
\smallskip
\begin{lem}[Subdifferentials under symmetry]
Consider a function $f\colon\R^n\to\overline{\R}$ that is locally symmetric at $\bar{x}$. Then the equation
$$\partial f(\sigma x)=\sigma \partial f(x),\quad \textrm{ holds for any } \sigma\in \fix(\bar{x}) \textrm{ and all } x \textrm{ near }\bar{x}.$$
Similarly, in terms of the spectral function $F:=f\circ\lambda$, we have
$$\partial F(U.X)=U.(\partial F(X)),\quad \textrm{ for any } U\in {\bf O}^n.$$
\end{lem}
\smallskip
\begin{rem}\label{rem:sym_sub}
{\rm
In particular, if $f\colon\R^n\to\overline{\R}$ is locally symmetric around $\bar{x}$, then the sets $\hat{\partial} f(\bar{x})$, $\ri \hat{\partial} f(\bar{x})$, $\rb \hat{\partial} f(\bar{x})$, $\aff \hat{\partial} f(\bar{x})$, and $\para \hat{\partial} f(\bar{x})$ are invariant under the action of the group $\fix  (\bar{x})$.}
\end{rem}

\smallskip
The following result is the cornerstone for the variational theory of spectral mappings \cite[Theorem 6]{lag}.

\smallskip
\begin{thm}[Subdifferential under local symmetry]
Consider a lsc function $f\colon\R^n\to\overline{\R}$ and a symmetric matrix $X\in {\bf S}^{n}$, and suppose that $f$ is locally symmetric at $\lambda(X)$. Then we have
$$\partial (f\circ \lambda)(X)=\{U^T (\Diag v) U: v\in \partial f(\lambda(X)) \textrm{ and } U\in {\bf O}^n_X\},$$
where
$${\bf O}^n_X=\{U\in {\bf O}^n: X=U^T(\Diag \lambda(X))U\}. $$
\end{thm}

It is often useful to require a certain uniformity of the
subgradients of the function. This is the content of the following
definition \cite[Definition~1.1]{prox_reg}.

\smallskip
\begin{defn}[Directional prox-regularity]
{\rm A function $f\colon\E\to\overline{\R}$ is called {\em
prox-regular at} $\bar{x}$ {\em for} $\bar{v}$, with
$\bar{v}\in\partial f(\bar{x})$, if $f$ is locally lsc at $\bar{x}$
and there exist $\epsilon >0$ and $\rho >0$ so that the inequality
$$f(y)\geq f(x)+\langle v,y-x\rangle -\frac{\rho}{2}\|y-x\|^2,$$ holds
whenever $x,y\in B_{\epsilon}(\bar{x})$, $v\in
B_{\epsilon}(\bar{v})\cap\partial f(x)$, and
$f(x)<f(\bar{x})+\epsilon$.}

\noindent {\rm The function $f$ is called {\em prox-regular at}
$\bar{x}$, if it is finite at $\bar{x}$ and $f$ is prox-regular at
$\bar{x}$ for every subgradient $v\in
\partial f(\bar{x})$.}

\noindent{\rm A set $Q\subset\R^n$ is {\em prox-regular at
}$\bar{x}$ {\em for} $\bar{v}\in N_Q(\bar{x})$ provided that the
indicator function $\delta_Q$ is prox-regular at $\bar{x}$ for
$\bar{v}$. }
\end{defn}
\smallskip

In particular ${\bf C}^2$-smooth functions and lsc, convex functions are prox-regular at each of their points \cite[Example 13.30, Proposition 13.34]{VA}.
\smallskip

\begin{rem}
{\rm A set $Q\subset\E$ is prox-regular at $\bar{x}$, in the sense above, if and only if it is prox-regular in the sense of Definition~\ref{defn:prox_set}. For a proof, see for example \cite[Exercise 13.38]{VA}.
}
\end{rem}
\smallskip

The following theorem shows that directional prox-regularity also satisfies the Transfer Principle \cite[Theorem 4.2]{spec_prox}.

\smallskip
\begin{thm}[Directional prox-regularity under spectral lifts]\label{thm:prox_lift} \hfill {\\}
Consider a lsc function $f\colon\R^n\to\overline{\R}$ and a symmetric matrix $\bar{X}$.
Suppose that $f$ is locally symmetric around $\bar{x}:=\lambda(\overline{X})$. Then $f$
is prox-regular at $\bar{x}$ if and only if $f\circ \lambda$ is prox-regular at $\overline{X}$.
\end{thm}
\smallskip

The following two standard results of Linear Algebra will be important for us \cite[Proposition 3]{lag}.
\smallskip

\begin{lem}[Simultaneous Conjugacy]\label{lem:sim_conj}
Consider vectors $x,y,u,v\in\R^n$. Then there exists an orthogonal matrix $U\in{\bf O}^n$ with
$$\Diag x= U^{T}(\Diag u)U \quad{\textrm and }\quad \Diag y =U^{T}(\Diag v) U,$$
if and only if there exists a permutation $\sigma\in\Sigma^n$ with $x=\sigma u$ and $y=\sigma v$.
\end{lem}
\smallskip

\begin{cor}[Conjugations and permutations]\label{cor:conjug2}
Consider vectors $v_1,v_2\in\R^n$ and a matrix $X\in{\bf S}^n$.
Suppose that for some $U_1,U_2\in {\bf O}_X^n$ we have
$$U_1^{T}(\Diag v_1)U_1=U_2^{T}(\Diag v_2)U_2.$$
Then there exists a permutation $\sigma\in \fix(\lambda(X))$
satisfying $\sigma v_1=v_2$.
\end{cor}
\begin{proof}
Observe $$(U_1 U_2^{T})^{T}\Diag v_1 (U_1 U_2^{T})=\Diag v_2,$$
$$(U_1 U_2^{T})^{T}\Diag \lambda(X) (U_1 U_2^{T})=\Diag \lambda(X).$$
The result follows by an application of Lemma~\ref{lem:sim_conj}.
\end{proof}

\subsection{Main results}
In this section, we consider partly-smooth sets, introduced in \cite{Lewis-active}.
This notion generalizes the idea of active manifolds of classical nonlinear programming
to an entirely representation-independent setting.
\smallskip

\begin{defn}[Partial Smoothness]
{\rm Consider a function $f\colon\E\to\overline{\R}$ and a set
$M\subset\E$ containing a point $\bar{x}$. Then $f$ is ${\bf
C}^p${\em -partly smooth} ($p=2,\ldots,\infty$) at $\bar{x}$
relative to $M$ if
\begin{itemize}
\item[{\rm (i)}] {\bf (Smoothness)} $M$ is a ${\bf C}^p$ manifold around $\bar{x}$ and $f$ restricted to $M$ is ${\bf C}^p$-smooth near $\bar{x}$,
\item[{\rm (ii)}] {\bf (Regularity)} $f$ is prox-regular at $\bar{x}$,
\item[{\rm (iii)}]{\bf (Sharpness)} the affine span of $\partial f$ is a translate of $N_M(x)$,
\item[{\rm (iv)}] {\bf (Continuity)} $\partial f$ restricted to $M$ is continuous at $\bar{x}$.
\end{itemize}
If the above properties hold, then we will refer to $M$ as the {\em partly smooth manifold} of $f$ at $\bar{x}$.
}
\end{defn}
\smallskip

\begin{rem}
{\rm Though the original definition of partial smoothness replaces
the prox-regularity condition by Clarke-regularity, we feel that the
prox-regularity is essential for the theory. In particular, without
it, partly-smooth manifolds are not even guaranteed to be locally
unique and the basic property of identifiability may fail
\cite[Section~7]{Hare}.}
\end{rem}
\smallskip

Some comments are in order. First the continuity property of
$\partial f$ is meant in the Painlev\'{e}-Kuratowski sense. See for
example \cite[Definition 5.4]{VA}. The exact details of this notion
will not be needed in our work, and hence we do not dwell on it
further. Geometrically, partly smooth manifolds have a
characteristic property in that the epigraph of $f$ looks
``valley-like'' along the graph of $f\big|_M$ . See Figure~\ref{fig:val} for an
illustration.

\begin{figure}[h]
  \centering\includegraphics[scale=0.5]{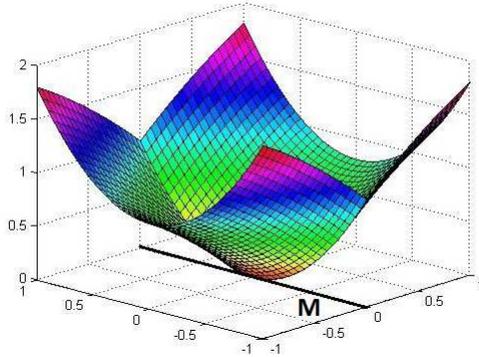}
  \caption{The partly smooth manifold $M$ for $f(x,y):=|x|(1-|x|)+y^2$.}\label{fig:val}
\end{figure}

It is reassuring to know that partly smooth manifolds are locally
unique. This is the content of the following theorem
\cite[Corollary~4.2]{Hare}.
\smallskip
\begin{thm}[Local uniqueness of partly smooth manifolds]\label{thm:uniqps}
Consider a function $f\colon\E\to\overline{\R}$ that is ${\bf C}^p$-partly smooth ($p\geq 2$) at $\bar{x}$ relative to two manifolds $M_1$ and $M_2$. Then there exists a neighborhood $U$ of $\bar{x}$ satisfying
$U\cap M_1=U\cap M_2$.
\end{thm}
\smallskip

Our goal in this section is to prove that partly smooth manifolds
satisfy the Transfer Principle. However, proving this directly is
rather difficult. This is in large part because the continuity of
the subdifferential mapping $\partial (f\circ\lambda)$ seems to be
intrinsically tied to continuity properties of the mapping
$$X\mapsto {\bf O}^n_X=\{U\in {\bf O}^n: X=U^T(\Diag \lambda(X))U\},$$
which are rather difficult to understand.

We however will side-step this problem entirely by instead focusing
on a property that is seemingly different from partial smoothness
--- {\em finite identification}. This notion is of significant
independent interest. It has been implicitly considered by a number
of authors in connection with the possibility to accelerate various
first-order numerical methods
\cite{Wright,Dunn87,Calamai-More87,Burke-More88,Burke90,Ferris91,Al-Khayyal-Kyparisis91,Flam92,
ident_bundle}, and has explicitly been studied in \cite{ident} for
its own sake.

\smallskip
\begin{defn}[Identifiable sets]
{\rm Consider a function $f\colon\E\to\overline{\R}$, a point $\bar{x}\in\R^n$,
and a subgradient $\bar{v}\in\partial f(\bar{x})$. A set $M\subset\dom f$ is
{\em identifiable at} $\bar{x}$ {\em for} $\bar{v}$ if for any sequences $(x_i,f(x_i),v_i)\to(\bar{x},f(\bar{x}),\bar{v})$,
with $v_i\in\partial f(x_i)$, the points $x_i$ must all lie in $M$ for all sufficiently large indices $i$.}
\end{defn}
\smallskip

\begin{rem}
{\rm
It is important to note that identifiable sets are not required to be smooth manifolds.
Indeed, as we will see shortly, identifiability is a more basic notion than partial smoothness.}
\end{rem}

The relationship between partial smoothness and finite
identification is easy to explain. Indeed, as the following theorem
shows, partial smoothness is in a sense just a ``uniform'' version
of identifiability \cite[Proposition 9.4]{ident}.

\smallskip
\begin{prop}[Partial smoothness and identifiability]\label{prop:eqv}
Consider a lsc function $f\colon\E\to\overline{\R}$ that is prox-regular
at a point $\bar{x}$. Let $M\subset\dom f$ be a ${\bf C}^p$ manifold ($p=2,\ldots,\infty$) containing $\bar{x}$, with the restriction $f\big|_M$ being ${\bf C}^p$-smooth near $\bar{x}$. Then the following are equivalent
\begin{enumerate}
\item $f$ is ${\bf C}^p$-partly smooth at $\bar{x}$ relative to $M$
\item $M$ is an identifiable set (relative to $f$) at $\bar{x}$ for every subgradient $\bar{v}\in\ri\partial f(\bar{x})$.
\end{enumerate}
\end{prop}
\smallskip

In light of the theorem above, our strategy for proving the Transfer Principle for partly smooth sets
is two-fold: first prove the analogous result for identifiable sets and then gain a better understanding
of the relationship between the  sets $\ri \partial f(\lambda(X))$ and $\ri \partial (f\circ\lambda)(X)$.


\smallskip
\begin{prop}[Spectral lifts of Identifiable sets]\label{prop:lift_id}
Consider a lsc function $f\colon\R^n\to\overline{\R}$ and a symmetric matrix $\overline{X}\in {\bf S}^n$. Suppose that $f$ is locally symmetric around $\bar{x}:=\lambda(\overline{X})$ and consider a subset $M\subset \R^n$ that is locally symmetric around $\bar{x}$.
Then $M$ is identifiable (relative to $f$) at $\bar{x}$ for $\bar{v}\in\partial f(\bar{x})$, if and only if $\lambda^{-1}(M)$ is identifiable (relative to $f\circ\lambda$) at $\overline{X}$ for $U^{T}(\Diag \bar{v}) U\in\partial (f\circ \lambda)(\overline{X})$, where $U\in {\bf O}^n_{\overline{X}}$ is arbitrary.
\end{prop}
\begin{proof} We first prove the forward implication. Fix a subgradient $$\overline{V}:=\overline{U}^{T}(\Diag \bar{v}) \overline{U}\in\partial (f\circ \lambda)(\overline{X}),$$ for an arbitrary transformation $\overline{U}\in {\bf O}^n_{\overline{X}}$. For convenience, let $F:=f\circ\lambda$ and consider a sequence $(X_i,F(X_i),V_i)\to(\overline{X},F(\overline{X}),\overline{V})$. Our goal is to show that for all large indices $i$, the inclusion $\lambda(X_i)\in M$ holds. To this end, there exist matrices $U_i\in {\bf O}^n_{X_i}$ and subgradients $v_i\in\partial f(\lambda(X_i))$ with
$$U_i^{T}(\Diag \lambda(X_i)) U_i=X_i \quad\textrm{ and }\quad U_i^{T}(\Diag v_i) U_i=V_i.$$
Restricting to a subsequence, we may assume that there exists a matrix $\widetilde{U}\in {\bf O}^n_{\overline{X}}$ satisfying $U_i\to\widetilde{U}$, and consequently there exists a subgradient $\tilde{v}\in\partial f(\lambda(\overline{X}))$ satisfying $v_i\to\tilde{v}$. Hence we obtain
$$\widetilde{U}^{T}(\Diag \lambda(\overline{X}))\widetilde{U}=\overline{X}\quad \textrm{and}\quad \widetilde{U}^{T}(\Diag \tilde{v})\widetilde{U}=\overline{V}=\overline{U}^{T}(\Diag \bar{v}) \overline{U}.$$

By Corollary~\ref{cor:conjug2}, there exists a permutation $\sigma\in\fix(\bar{x})$ with $\sigma\tilde{v}=\bar{v}$.
Observe $(\lambda(X_i),f(\lambda(X_i)),v_i)\to(\bar{x},f(\bar{x}),\tilde{v})$. Observe that the set $\sigma^{-1}M$ is identifiable (relative to $f$) at $\bar{x}$ for $\tilde{v}$. Consequently for all large indices $i$, the inclusion $\lambda(X_i)\in \sigma^{-1}M$ holds. Since $M$ is locally symmetric at $\bar{x}$, we deduce that all the points $\lambda(X_i)$ eventually lie in $M$.

To see the reverse implication, fix an orthogonal matrix ${\overline{U}\in\bf O}_{\overline{X}}^n$ and define $\overline{V}:=\overline{U}^{T}(\Diag \bar{v}) \overline{U}$. Consider a sequence $(x_i,f(x_i),v_i)\to(\bar{x},f(\bar{x}),\bar{v})$ with $v_i\in\partial f(x_i)$.
It is not difficult to see then that there exist permutations $\sigma_i\in\fix(\bar{x})$ satisfying $\sigma_i x_i\in\R_{\geq}$. Restricting to a subsequence, we may suppose that $\sigma_i$ are equal to a fixed $\sigma\in\fix(\bar{x})$. Define
$$X_i:=\overline{U}^{T}(\Diag \sigma x_i) \overline{U} \quad \textrm{and} \quad V_i:=\overline{U}^{T}(\Diag \sigma v_i)\overline{U}.$$ Letting $A_{\sigma^{-1}}\in{\bf O}^n$ denote the matrix representing the permutation $\sigma^{-1}$, we have
\begin{align*}
X_i&:=(\overline{U}^{T}A_{\sigma^{-1}}\overline{U})^{T}\big[\overline{U}^{T}(\Diag x_i) \overline{U}\big]\overline{U}^{T} A_{\sigma^{-1}}\overline{U} \quad
\textrm{and} \\
\quad V_i&:= (\overline{U}^{T}A_{\sigma^{-1}}\overline{U})^{T}[\overline{U}^{T}(\Diag  v_i)\overline{U}] \overline{U}^{T}A_{\sigma^{-1}}\overline{U}.
\end{align*}
We deduce $X_i\to (\overline{U}^{T}A_{\sigma^{-1}}\overline{U})^{T}\overline{X}(\overline{U}^{T} A_{\sigma^{-1}}\overline{U})$ and $V_i\to (\overline{U}^{T}A_{\sigma^{-1}}\overline{U})^{T}\overline{V}(\overline{U}^{T} A_{\sigma^{-1}}\overline{U})$.
On the other hand, observe $\overline{X}=(\overline{U}^{T}A_{\sigma^{-1}}\overline{U})^{T}\overline{X}(\overline{U}^{T} A_{\sigma^{-1}}\overline{U})$. Since $\lambda^{-1}(M)$ is identifiable (relative to $F$) at $\overline{X}$ for $(\overline{U}^{T}A_{\sigma^{-1}}\overline{U})^{T}\overline{V}(\overline{U}^{T} A_{\sigma^{-1}}\overline{U})$, we deduce that the matrices $X_i$ lie in $\lambda^{-1}(M)$ for all sufficiently large indices $i$. Since $M$ is locally symmetric around $\bar{x}$, the proof is complete.
\end{proof}
\smallskip


Using the results of Section~\ref{sec:lift_man}, we can now describe in a natural way the affine span, relative interior, and relative boundary of the Fr\'{e}chet subdifferential. We begin with a lemma.
\smallskip
\begin{lem}[Affine generation]\label{lem:aff_gen}
Consider a matrix $X\in {\bf S}^n$ and suppose that the point $x:=\lambda(X)$ lies in an affine subspace $\mathcal{V}\subset\R^n$ that is invariant under the action of $\fix (x)$. Then the set
$$\{U^T (\Diag v) U: v\in \mathcal{V} \textrm{ and } U\in {\bf O}^n_X\},$$
is an affine subspace of ${\bf S}^n$.
\end{lem}
\begin{proof}
Define the set $L:=(\para \mathcal{V})^{\perp}$.
Observe that the set $L\cap \mathcal{V}$ consists of a single vector; call this vector $w$. Since both $L$ and $\mathcal{V}$ are invariant under the action of $\fix(x)$, we deduce $\sigma w=w$ for all $\sigma\in \fix(x)$.

Now define a function $g\colon\R^n\to\overline{\R}$ by declaring $$g(y)=\langle w,y\rangle+\delta_{x+L}(y),$$
and note that the equation $$\hat{\partial} g(x):= w+N_{x+L}(x)=\mathcal{V}, \quad\textrm{ holds}.$$
Observe that for any permutation $\sigma\in\fix(x)$, we have $$g(\sigma y)=\langle w,\sigma y\rangle +\delta_{x+L}(\sigma y)=\langle \sigma^{-1}w, y\rangle +\delta_{x+\sigma^{-1}L}(y)=g(y).$$ Consequently $g$ is locally symmetric at $x$. Observe
$$(g\circ\lambda)(Y)=\langle w,\lambda(Y)\rangle+\delta_{\lambda^{-1}(x+L)}{Y}.$$
It is immediate from Theorems~\ref{thm:sm_lift} and \ref{thm:lift_man}, that the function $Y\mapsto \langle w,\lambda(Y)\rangle$ is ${\bf C}^{\infty}$-smooth around $X$ and that $\lambda^{-1}(x+L)$ is a ${\bf C}^{\infty}$ manifold around $X$. Consequently $\hat{\partial} (g\circ\lambda)(X)$ is an affine subspace of ${\bf S}^n$.
On the other hand, we have
$$\hat{\partial} (g\circ \lambda)(X)=\{U^T (\Diag v) U: v\in \mathcal{V} \textrm{ and } U\in {\bf O}^n_X\},$$
thereby completing the proof.
\end{proof}

\smallskip

\begin{prop}[Affine span of the spectral Fr\'{e}chet subdifferential]\label{prop:ri_sub}
Consider a function $f\colon\R^n\to\overline{\R}$ and a matrix $X\in{\bf S}^n$. Suppose that $f$ is locally symmetric at $\lambda(X)$. Then we have
\begin{align}
\aff \hat{\partial} (f\circ\lambda) (X)&=\{U^T (\Diag v) U: v\in \aff\hat{\partial} f(\lambda(X)) \textrm{ and } U\in {\bf O}^n_X\},\label{eqn:aff}\\
\rb \hat{\partial} (f\circ\lambda) (X)&=\{U^T (\Diag v) U: v\in \rb\hat{\partial} f(\lambda(X)) \textrm{ and } U\in {\bf O}^n_X\}.\label{eqn:rb}\\
\ri \hat{\partial} (f\circ\lambda) (X)&=\{U^T (\Diag v) U: v\in \ri\hat{\partial} f(\lambda(X)) \textrm{ and } U\in {\bf O}^n_X\}.\label{eqn:ri}
\end{align}
\end{prop}
\begin{proof}
Throughout the proof, let $x:=\lambda(X)$. We prove the formulas in the order that they are stated. To this end, observe that the inclusion $\supset$ in (\ref{eqn:aff}) is immediate.
Furthermore, the inclusion
$$\hat{\partial} (f\circ\lambda) (X)\subset\{U^T (\Diag v) U: v\in \aff\hat{\partial} f(\lambda(X)) \textrm{ and } U\in {\bf O}^n_X\}.$$
clearly holds. Hence to establish the reverse inclusion in
(\ref{eqn:aff}), it is sufficient to show that the set $$\{U^T
(\Diag v) U: v\in \aff\hat{\partial} f(\lambda(X)) \textrm{ and }
U\in {\bf O}^n_X\},$$ is an affine subspace; but this is immediate
from Remark~\ref{rem:sym_sub} and Lemma~\ref{lem:aff_gen}. Hence
(\ref{eqn:aff}) holds.

\noindent We now prove (\ref{eqn:rb}). Consider a matrix $U^T(\Diag
v) U\in \rb \hat{\partial}(f\circ \lambda)(X)$ with $U\in {\bf
O}_X^n$ and $v\in \hat{\partial} f(\lambda(X))$. Our goal is to show
the stronger inclusion $v\in \rb \hat{\partial} f(x)$. Observe from
(\ref{eqn:aff}), there exists a sequence $U^{T}_i (\Diag v_i) U_i\to
U^T(\Diag v) U$ with $U_i\in {\bf O}_X^n$, $v_i\in
\aff\hat{\partial} f(x)$, and $v_i\notin \hat{\partial} f(x)$.
Restricting to a subsequence, we may assume that there exists a
matrix $\widetilde{U}\in {\bf O}_X^n$ with $U_i\to \widetilde{U}$
and a vector $\tilde{v}\in \aff\hat{\partial} f(x)$ with $v_i\to
\tilde{v}$. Hence the equation
$$\widetilde{U}^T(\Diag \tilde{v}) \widetilde{U}=U^T(\Diag v) U,\quad \textrm{holds}.$$
Consequently, by Corollary~\ref{cor:conjug2}, there exists a permutation $\sigma\in \fix(x)$ satisfying $\sigma \tilde{v}=v$. Since $\hat{\partial} f(x)$ is invariant under the action of $\fix(x)$, it follows that $\tilde{v}$ lies in $\rb \hat{\partial} f(x)$, and consequently from Remark~\ref{rem:sym_sub} we deduce $v\in\rb \hat{\partial} f(x)$. This establishes the inclusion $\subset$ of (\ref{eqn:rb}). To see the reverse inclusion, consider a sequence $v_i\in\aff\hat{\partial} f(x)$ converging to $v\in\hat{\partial} f(x)$ with $v_i\notin\hat{\partial} f(x)$ for each index $i$. Fix an arbitrary matrix $U\in {\bf O}_X^n$ and observe that the matrices $U^{T}(\Diag v_i) U$ lie in $\aff \hat{\partial} (f\circ\lambda)(x)$ and converge to $U^{T}(\Diag v) U$. We now claim that the matrices $U^{T}(\Diag v_i) U$  all lie outside of $\hat{\partial} (f\circ\lambda)(x)$. Indeed suppose this is not the case. Then there exist matrices $\widetilde{U}_i\in {\bf O}^n_{X}$ and subgradients $v_i\in\hat{\partial} f(x)$ satisfying
$$U^{T}(\Diag v_i) U= \widetilde{U}_i^{T}(\Diag \tilde{v}_i) \widetilde{U}_i.$$ An application of Corollary~\ref{cor:conjug2} and Remark~\ref{rem:sym_sub} then yields a contradiction. Therefore the inclusion $U^{T}(\Diag v) U\in\rb \hat{\partial} (f\circ \lambda)(X)$ holds, and the validity of (\ref{eqn:rb}) follows.

Finally, we aim to prove (\ref{eqn:ri}). Observe that the inclusion
$\subset$ of  $(\ref{eqn:ri})$ is immediate from equation
$(\ref{eqn:rb})$. To see the reverse inclusion, consider a matrix
$U^{T}(\Diag v)U$, for some $U\in {\bf O}^n_X$ and $v\in\ri
\hat{\partial} f(x)$. Again, an easy application of
Corollary~\ref{cor:conjug2} and Remark~\ref{rem:sym_sub} yields the
inclusion $U^{T}(\Diag v)U\in\ri \hat{\partial} (f\circ\lambda)(X)$.
We conclude that (\ref{eqn:ri}) holds.
\end{proof}
\smallskip


\begin{lem}[Symmetry of partly smooth manifolds]\label{lem:ps_man_sym}
Consider a lsc function $f\colon\R^n\to\overline{\R}$ that is locally symmetric at $\bar{x}$. Suppose that $f$ is ${\bf C}^p$-partly smooth at $\bar{x}$ relative to $M$. Then $M$ is locally symmetric around $\bar{x}$.
\end{lem}
\begin{proof}
Consider a permutation $\sigma\in\fix(\bar{x})$. Then the function $f$ is partly smooth at $\bar{x}$ relative to $\sigma M$. On the other hand, partly smooth manifolds are locally unique Theorem~\ref{thm:uniqps}. Consequently we deduce equality $M=\sigma M$ locally around $\bar{x}$. The claim follows.
\end{proof}
\smallskip

The main result of this section is now immediate.
\smallskip
\begin{thm}[Lifts of ${\bf C}^{\infty}$-partly smooth functions]\label{thm:lift_pman}
Consider a lsc function $f\colon\R^n\to\overline{\R}$ and a matrix $\overline{X}\in{\bf S}^n$. Suppose that $f$ is locally symmetric around $\bar{x}:={\lambda(\overline{X})}$. Then $f$ is ${\bf C}^{\infty}$-partly smooth at $\bar{x}$ relative to $M$ if and only if $f\circ\lambda$ is ${\bf C}^{\infty}$-partly smooth at $\overline{X}$ relative to $\lambda^{-1}(M)$.
\end{thm}
\begin{proof}
Suppose that $f$ is ${\bf C}^{\infty}$-partly smooth at $\bar{x}$ relative to $M$. In light of Lemma~\ref{lem:ps_man_sym}, we deduce that $M$ is locally symmetric at $\bar{x}$. Consequently, Theorem~\ref{thm:lift_man} implies that the set $\lambda^{-1}(M)$ is a ${\bf C}^{\infty}$ manifold, while Corollary~\ref{cor:rest_smooth} implies that $f\circ\lambda$ is ${\bf C}^{\infty}$-smooth on $\lambda^{-1}(M)$ near $\overline{X}$. Applying Theorem~\ref{thm:prox_lift}, we conclude that $f\circ\lambda$ is prox-regular at $\bar{X}$. Consider now a subgradient $V\in\ri\partial (f\circ\lambda)(\overline{X})$. Then by Proposition~\ref{prop:ri_sub}, there exists a vector $v\in\ri \partial f(\bar{x})$ and a matrix $U\in{\bf O}^n_{\overline{X}}$ satisfying
$$V=U^T(\Diag v)U \quad\textrm{ and }\quad \overline{X}=U^T(\Diag \bar{x})U.$$
Observe by Proposition~\ref{prop:eqv}, the set $M$ is identifiable at $\bar{x}$ for $\bar{v}$. Then applying Proposition~\ref{prop:lift_id}, we deduce that $\lambda^{-1}(M)$ is identifiable (relative to $f\circ\lambda$) at $\overline{X}$ relative to $V$. Since $V$ is an arbitrary element of $\ri\partial (f\circ\lambda)(\overline{X})$, applying Proposition~\ref{prop:eqv}, we deduce that $f\circ\lambda$ is ${\bf C}^{\infty}$-partly smooth at $\overline{X}$ relative to $\lambda^{-1}(M)$, as claimed. The converse follows along the same lines.
\end{proof}
\smallskip

The forward implication of Theorem~\ref{thm:lift_pman} holds in the
case of ${\bf C}^p$-partly smooth functions (for
$p=2,\ldots,\infty$). The proof is identical except one needs to use
\cite[Theorem 4.21]{man} instead of Theorem~\ref{thm:lift_man}. We
record this result for ease of reference in future works.

\smallskip
\begin{thm}[Lifts of ${\bf C}^{p}$-partly smooth functions]
Consider a lsc function $f\colon\R^n\to\overline{\R}$ and a matrix $\overline{X}\in{\bf S}^n$.
Suppose that $f$ is locally symmetric around $\bar{x}:={\lambda(\overline{X})}$. If $f$ is
${\bf C}^{p}$-partly smooth (for $p=2,\ldots,\infty$) at $\bar{x}$ relative to $M$, then
$f\circ\lambda$ is ${\bf C}^{\infty}$-partly smooth at $\overline{X}$ relative to $\lambda^{-1}(M)$.
\end{thm}

\section{Partly smooth duality for polyhedrally generated spectral functions}\label{sec:dual}
Consider a lsc, convex function $f\colon\E\to\overline{\R}$. Then
the {\em Fenchel conjugate} $f^{*}\colon\E\to\overline{\R}$ is
defined by setting $$f^{*}(y)=\sup_{x\in\R^n} \{\langle x, y\rangle
- f(x)\}.$$ Moreover, in terms of the powerset of $\E$, denoted
$\mathbb{P}(\E)$, we define a correspondence
$\mathcal{J}_f\colon\mathbb{P}(\E)\to\mathbb{P}(\E)$ by setting
$$\mathcal{J}_f(Q):=\bigcup_{x\in Q}\ri\partial f(x).$$ The
significance of this map will become apparent shortly. Before
proceeding, we recall some basic properties of the conjugation
operation:
\smallskip
\begin{description}
\item[Biconjugation:]\quad $f^{**}=f$,
\item[Subgradient inversion formula:]\quad $\partial f^{*}=(\partial f)^{-1}$,
\item[Fenchel-Young inequality:]\quad $\langle x,y\rangle\leq f(x)+f^{*}(y)$ for every $x,y\in\R^n$.
\end{description}
\smallskip

Moreover, convexity and conjugation behave well under spectral
lifts. See for example \cite[Section~5.2]{cov_lift}.
\smallskip
\begin{thm}[Lifts of convex sets and conjugation]
If $f\colon\R^n\to\overline{\R}$ is a symmetric function, then $f^{*}$ is also symmetric
and the formula $$(f\circ \lambda)^{*}=f^{*}\circ\lambda, \quad \textrm{holds}.$$
Furthermore $f$ is convex if and only if the spectral function $f\circ \lambda$ is convex.
\end{thm}
\smallskip

The following definition is standard.
\smallskip
\begin{defn}[Stratification]
{\rm A finite partition $\mathcal{A}$ of a set $Q\subset\E$ is a
{\em stratification} provided that for any partitioning sets (called
{\em strata}) $M_1$ and $M_2$ in $\mathcal{A}$, the implication
$$M_1\cap \cl M_2\neq \emptyset\quad \Longrightarrow \quad
M_1\subset \cl M_2,\quad \textrm{ holds}.$$ If the strata are open
polyhedra, then $\mathcal{A}$ is a {\em polyhedral stratification}.
If the strata are ${\bf C}^k$ manifolds, then $\mathcal{A}$ is a
${\bf C}^k$-{\em stratification}. }
\end{defn}

\bigskip
{\bf Stratification duality for convex polyhedral functions.} We now
establish the setting and notation for the rest of the section.
Suppose that $f\colon\R^n\to\overline{\R}$ is a convex {\em
polyhedral} function (epigraph of $f$ is a closed convex
polyhedron). Then $f$ induces a finite polyhedral stratification
$\mathcal{A}_f$ of $\dom f$ in a natural way. Namely, consider the
partition of $\epi f$ into open faces $\{F_i\}$. Projecting all
faces $F_i$, with $\dim F_i\leq n$, onto the first $n$-coordinates
we obtain a stratification of the domain $\dom f$ of $f$ that we
denote by $\mathcal{A}_f$. In fact, one can easily see that $f$ is
${\bf C}^{\infty}$-partly smooth relative to each polyhedron $M\in
\mathcal{A}_f$.

A key observation for us will be that the correspondence $f\xleftrightarrow{*}f^{*}$
is not only a pairing of functions, but it also induces a duality pairing between
$\mathcal{A}_f$ and $\mathcal{A}_{f^{*}}$. Namely, one can easily check that the
mapping $\mathcal{J}_f$ restricts to an invertible  mapping $\mathcal{J}_f\colon\mathcal{A}_f\to \mathcal{A}_{f^{*}}$
with inverse given by $\mathcal{J}_{f^{*}}$.


\bigskip
{\bf Limitations of stratification duality.} It is natural to ask
whether for general (nonpolyhedral) lsc, convex functions
$f\colon\R^n\to\overline{\R}$, the correspondence
$f\xleftrightarrow{*}f^{*}$, along with the mapping $\mathcal{J}$,
induces a pairing between {\em partly smooth manifolds} of $f$ and
$f^{*}$. Little thought, however shows an immediate obstruction:
images of ${\bf C}^{\infty}$-smooth manifolds under the map
$\mathcal{J}_f$ may fail to be even ${\bf C}^2$-smooth.
\smallskip
\begin{exa}[Failure of smoothness]
{\rm Consider the conjugate pair $$f(x,y)=\frac{1}{4}(x^4+y^4)
\quad\quad\textrm{and}\quad\quad
f^{*}(x,y)=\frac{3}{4}(|x|^\frac{4}{3}+|y|^\frac{4}{3}).$$ Clearly
$f$ is partly smooth relative to $\R^2$, whereas any possible
partition of $\R^2$ into partly smooth manifolds relative to $f^{*}$
must consist of at least three manifolds (one mani\-fold in each
dimension: one, two, and three). Hence no duality pairing between
partly smooth manifolds is possible. See the Figures~\ref{fig:figure2} and \ref{fig:figure3} for an
illustration.

\begin{figure}[h!]
\begin{minipage}[b]{0.45\linewidth}
\centering
\includegraphics[scale=0.5]{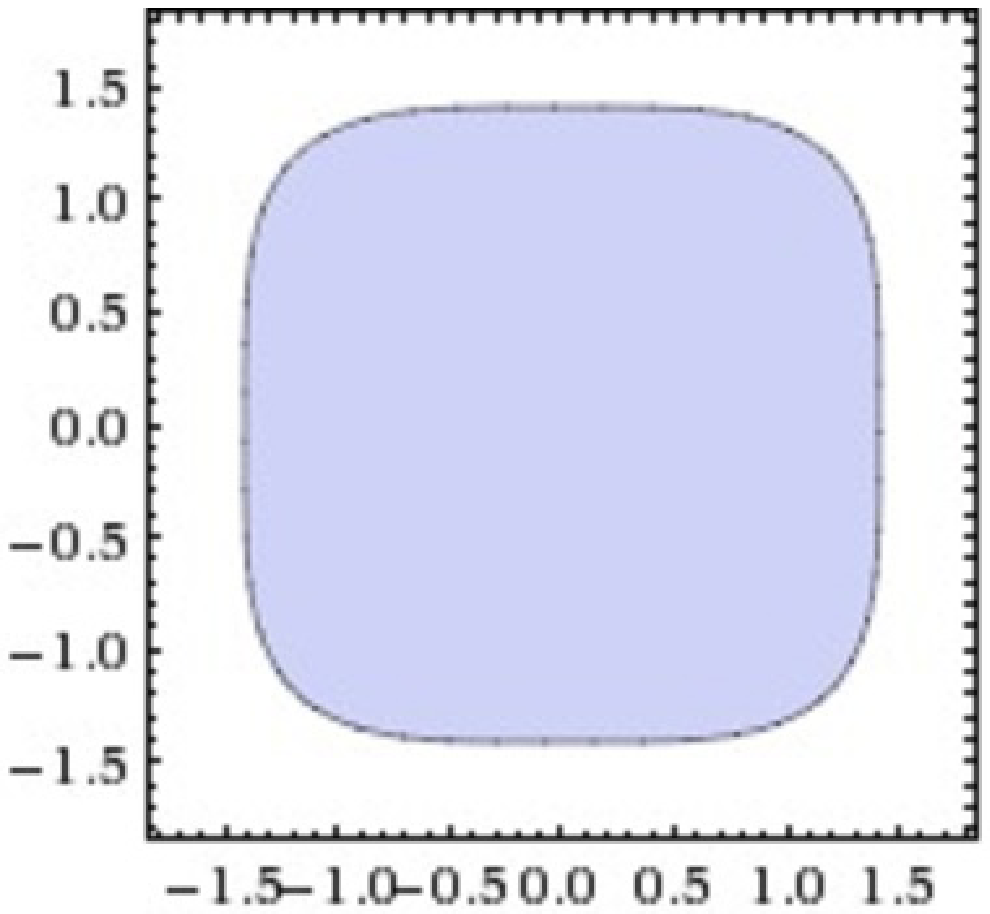}
\caption{$\{(x,y):x^4+y^4\leq 4\}$}
\label{fig:figure2}
\end{minipage}
\hspace{0.5cm}
\begin{minipage}[b]{0.45\linewidth}
\centering
\includegraphics[scale=0.475]{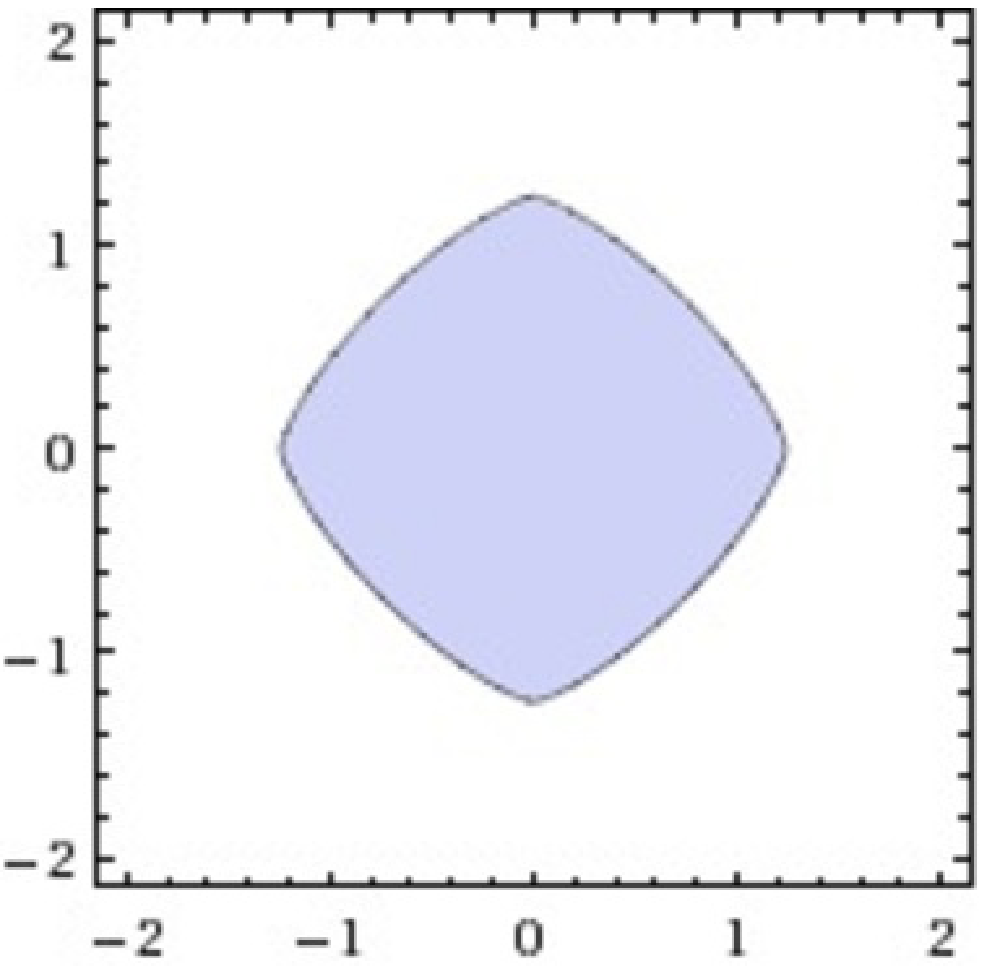}
\caption{$\{(x,y):|x|^\frac{4}{3}+|y|^\frac{4}{3}\leq \frac{4}{3}\}$}
\label{fig:figure3}
\end{minipage}
\end{figure}
}
\end{exa}
\smallskip

Indeed, this is not very surprising, since the convex duality is
really a duality between smoothness and {\em strict} convexity. See
for example \cite[Section 4]{cmd} or \cite[Theorem~11.13]{VA}. Hence
in general, one needs to impose tough strict convexity conditions in
order to hope for this type of duality to hold. Rather than doing
so, and more in line with the current work, we consider the spectral
setting. Namely, we will show that in the case of spectral functions
$F:=f\circ \lambda$, with $f$ symmetric and polyhedral --- functions
of utmost importance in eigenvalue optimization --- the mapping
$\mathcal{J}$ does induce a duality correspondence between partly
smooth manifolds of $F$ and $F^{\ast}$.

In the sequel, let us denote by
\[
M^{\mathrm{sym}}:=\bigcup_{\sigma \in \Sigma}\sigma M
\]
the symmetrization of any subset $M\subset \R^{n}$. Before we
proceed, we will need the following result.

\smallskip
\begin{lem}
[Path-connected lifts]\label{nhad}Let $M\subseteq \R^{n}$ be a path-connected set and assume that for any
permutation $\sigma \in \Sigma$, we either have $\sigma M=M$ or
$\sigma M\cap M=\emptyset$. Then $\lambda^{-1}(M^{\mathrm{sym}})$ is
a path-connected subset of $\mathbf{S}^{n}$.
\end{lem}

\begin{proof}
Let $X_{1},X_{2}$ be in $\lambda^{-1}(M^{\mathrm{sym}}),$
and set $x_{i}=\lambda(X_{i})\in M^{\mathrm{sym}}\cap
\R_{\geq}^{n},$ for $i\in \{1,2\}$. It is standard to check that the sets $\lambda^{-1}%
(x_{i})$ are path-connected manifolds for $i =1,2.$ Consequently the matrices $X_{i}$ and $\mathrm{Diag}(x_{i})$ can be
joined via a path lying in $\lambda
^{-1}(x_{i})$. Thus in order to construct a path joining $X_{1}$
to $X_{2}$ and lying in $\lambda^{-1}(M^{\mathrm{sym}})$ it would be
sufficient to join $x_{1}$ to $x_{2}$ inside $M^{\mathrm{sym}}$.
This in turn will follow immediately if both $\sigma x_{1},\sigma x_{2}$
belong in $M$ for some $\sigma \in \Sigma$. To establish this, we
will assume without loss of generality that $x_{1}$ lies in $M$. In
particular, we have $M\cap \R_{\geq}^{n}\neq \emptyset$ and we will establish
the inclusion $x_{2}\in M$.

To this end, consider a permutation $\sigma \in \Sigma$ satisfying $x_{2}\in \sigma
M\cap \R_{\geq }^{n}$. Our immediate goal is to establish $\sigma M\cap M\neq
\emptyset,$ and thus $\sigma M=M$ thanks to our assumption. To this end,
consider the point $y\in M$ satisfying $x_{2}=\sigma y$. If $y$ lies in $\R_{\geq}^{n}$,
then we deduce $y= x_2$ and we are done. Therefore, we can assume $y\notin \R_{\geq}^{n}$.
We can then consider the decomposition $\sigma=\sigma_{k}\cdots \sigma_{1}$ of
the permutation $\sigma$ into 2-cycles $\sigma_{i}$ each of which
permutes exactly two coordinates of $y$ that are not in the right
(decreasing) order. For the sake of brevity, we omit details of the construction of such a decomposition; besides, it is rather standard. We claim now $\sigma_{1}M=M$. To see this, suppose that
$\sigma _{1}$ permutes the $i$ and $j$ coordinates of $y$ where
$y_{i}<y_{j}$ and $i>j$. Since $x_1$ lies in $\R_{\geq}^{n}$ and $M$ is path-connected, there exists a point $z\in M$ satisfying $z_{i}=z_{j}$. Then
$\sigma_{1}z=z$, whence $\sigma_{1}M=M$ and $\sigma_{1}y\in M$.
Applying the same argument to $\sigma_{1}y$ and $\sigma_{1}M$ with
the 2-cycle $\sigma_{2}$ we obtain $\sigma_{2}\sigma _{1}M=M$ and
$\sigma_{2}\sigma_{1}y\in M$. By induction, $\sigma M=M$. Thus
$x_{2}\in M$ and the assertion follows.
\end{proof}

\bigskip
{\bf Stratification duality for spectral lifts.} Consider a symmetric, convex polyhedral function $f\colon\R^n\to\overline{\R}$
together with its induced stratification $\mathcal{A}_f$ of $\dom
f$. Then with each polyhedron $M\in \mathcal{A}_f$, we may associate
the symmetric set $M^{{\rm sym}}$. We record some properties of such
sets in the following lemma.

\smallskip
\begin{lem}[Properties of $\mathcal{A}_f$]\label{lem:perm_poly}
Consider a symmetric, convex polyhedral function $f\colon\R^n\to\overline{\R}$ and the induced stratification
$\mathcal{A}_f$ of $\dom f$. Then the following are true.
\begin{itemize}
\item[{\rm (i)}] For any set $M_1, M_2\in\mathcal{A}_f$ and any permutation $\sigma\in\Sigma$,
the sets $\sigma M_1$ and $M_2$ either coincide or are disjoint.
\item[{\rm (ii)}] The action of $\Sigma$ on $\R^n$ induces an action of $\Sigma$ on $$\mathcal{A}^k_f:=\{M\in\mathcal{A}_f: \dim M=k\}$$
for each $k=0,\ldots,n$. In particular, the set $M^{{\rm sym}}$
is simply the union of all polyhedra belonging to the orbit of $M$
under this action.
\item[{\rm (iii)}] For any polyhedron $M\in \mathcal{A}_f$, and every point $x\in M$,
there exists a neighborhood $U$ of $x$ satisfying $U\cap M^{{\rm
sym}}=U\cap M$. Consequently, $M^{{\rm sym}}$ is a ${\bf
C}^{\infty}$ manifold of the same dimension as $M$. \\
Moreover, $\lambda^{-1}(M^{\rm sym})$ is connected, whenever $M$ is.
\end{itemize}
\end{lem}
\smallskip

The last assertion follows from Lemma~\ref{nhad}. The remaining
assertions are straightforward and hence we omit their proof.

Notice that the strata of the stratification $\mathcal{A}_f$ are
connected ${\bf C}^{\infty}$ manifolds, which fail to be symmetric
in general. In light of Lemma~\ref{lem:perm_poly}, the set $M^{{\rm
sym}}$ is a ${\bf C}^{\infty}$ manifold and a disjoint union of open
polyhedra. Thus the collection $$\mathcal{A}^{\rm sym}_f:=\{M^{{\rm
sym}}: M\in \mathcal{A}_f\},$$ is a stratification of $\dom f$,
whose strata are now symmetric manifolds. Even though the new strata
are disconnected, they give rise to connected lifts
$\lambda^{-1}(M^{{\rm sym}})$. One can easily verify
that, as before, $\mathcal{J}_f$ restricts to an invertible mapping
$\mathcal{J}_f\colon\mathcal{A}^{\rm sym}_f\to \mathcal{A}^{\rm
sym}_{f^{*}}$ with inverse given by the restriction of
$\mathcal{J}_{f^{*}}$.

We now arrive at the main result of the section.
\smallskip

\begin{thm}[Lift of the duality map]\label{thm:liftdual}
Consider a symmetric, convex polyhedral function $f\colon\R^n\to\overline{\R}$ and define the spectral function $F:=f\circ\lambda$. Let $\mathcal{A}_f$ be the finite polyhedral partition of $\dom f$ induced by $f$, and define the collection $$\mathcal{A}_F:=\Big\{\lambda^{-1}(M^{{\rm sym}}): M\in\mathcal{A}_f\Big\}.$$
Then the following properties hold:
\begin{itemize}
\item[{\rm (i)}]
$\mathcal{A}_F$ is a ${\bf C}^{\infty}$-stratification of $\dom F$ comprised of connected manifolds,
\item[{\rm (ii)}] $F$ is ${\bf C}^{\infty}$-partly smooth relative to each set $\lambda^{-1}(M^{{\rm sym}})\in \mathcal{A}_F$.
\item[{\rm (iii)}] The assignment $\mathcal{J}_F\colon \mathbb{P}({\bf S}^n)\to\mathbb{P}({\bf S}^n)$
restricts to an invertible mapping
$\mathcal{J}_F\colon\mathcal{A}_F\to \mathcal{A}_{F^{*}}$ with
inverse given by the restriction of $\mathcal{J}_{F^{*}}$.
\item[{\rm (iv)}]
The following diagram commutes:
\begin{center}
\begin{tikzpicture}[node distance=2cm, auto]
  \node (AF) {$\mathcal{A}_F$};
  \node (AFs) [right of=AF] {$\mathcal{A}_{F^{*}}$};
  \node (Af) [below of=AF] {$\mathcal{A}^{{\rm sym}}_f$};
  \node (Afs) [below of=AFs] {$\mathcal{A}^{{\rm sym}}_{f^{*}}$};
  \draw[->] (AF) to node {$\mathcal{J}_F$} (AFs);
  \draw[->] (Af) to node [swap] {$\mathcal{J}_f$} (Afs);
  \draw[->] (Af) to node {$\lambda^{-1}$} (AF);
  \draw[->] (Afs) to node [swap] {$\lambda^{-1}$} (AFs);
\end{tikzpicture}
\end{center}
That is, the equation $(\lambda^{-1}\circ\mathcal{J}_f)(M^{{\rm sym}})=(\mathcal{J}_F\circ\lambda^{-1})(M^{{\rm sym}})$ holds for every set $M^{{\rm sym}}\in\mathcal{A}^{{\rm sym}}_f$.
\end{itemize}
\end{thm}
\begin{proof} In light of Lemma~\ref{lem:perm_poly}, each set $M^{{\rm
sym}}\in\mathcal{A}^{{\rm sym}}_f$ is a symmetric ${\bf C}^{\infty}$
manifold. The fact that $\mathcal{A}_F$ is a ${\bf
C}^{\infty}$-stratification of $\dom F$ now follows from the
transfer principle for stratifications \cite[Theorem~4.8]{mather}, while the fact that each manifold $\lambda^{-1}(M^{{\rm sym}})$ is connected follows immediately from Lemma~\ref{lem:perm_poly}.
Moreover, from Theorem~\ref{thm:lift_pman}, we deduce that $F$ is
${\bf C}^{\infty}$-partly smooth relative to each set in
$\mathcal{A}_F$.



Consider now a set $M^{{\rm sym}}\in\mathcal{A}^{{\rm sym}}_f$ for some  $M\in\mathcal{A}_f$.
Then we have:
\begin{align*}
\mathcal{J}_F(\lambda^{-1}(M^{{\rm sym}}))&=\bigcup_{X\in\lambda^{-1}(M^{{\rm sym}})}\ri \partial F(X)\\
&=\bigcup_{X\in \lambda^{-1}(M^{{\rm sym}})} \{U^T (\Diag v) U: v\in
\ri\partial f(\lambda(X)) \textrm{ and } U\in {\bf O}^n_X\},
\end{align*}
and concurrently,
$$\lambda^{-1}(\mathcal{J}_f(M^{{\rm sym}}))=\lambda^{-1}\Big(\bigcup_{x\in M^{{\rm sym}}} \ri\partial f(x)\Big)
=\bigcup_{x\in M^{{\rm sym}},~ v\in{\scriptsize \ri}\partial f(x)} {\bf O}^n.(\Diag v).$$

We claim that the equality $\lambda^{-1}(\mathcal{J}_f(M^{{\rm
sym}}))=\mathcal{J}_F(\lambda^{-1}(M^{{\rm sym}}))$ holds. The
inclusion ``$\supset$'' is immediate. To see the converse, fix a
point $x\in M^{{\rm sym}}$, a vector $v\in\ri\partial f(x)$, and a
matrix $U\in{\bf O}^n$. We must show $V:=U^T(\Diag v)U\in
\mathcal{J}_F(\lambda^{-1}(M^{{\rm sym}}))$. To see this, fix a
permutation $\sigma\in\Sigma$ with $\sigma x\in\R^n_{\geq}$, and
observe
$$U^T(\Diag v)U=(A_{\sigma} U)^T(\Diag \sigma v)A_{\sigma} U,$$
where $A_{\sigma}$ denotes the matrix representing the permutation $\sigma$.
Define a matrix $X:=(A_{\sigma} U)^T(\Diag \sigma x)A_{\sigma} U$. Clearly, we have $V\in \ri\partial F(X)$ and $X\in\lambda^{-1}(M^{{\rm sym}})$. This proves the claimed equality. Consequently, we deduce that
the assignment $\mathcal{J}_F\colon \mathbb{P}({\bf S}^n)\to\mathbb{P}({\bf S}^n)$
restricts to a mapping $\mathcal{J}_F\colon\mathcal{A}_F\to \mathcal{A}_{F^{*}}$, and that the diagram commutes. Commutativity of the diagram along with the fact that $\mathcal{J}_{f^{*}}$ restricts to be the inverse of  $\mathcal{J}_{f}\colon\mathcal{A}^{{\rm sym}}_f\to\mathcal{A}^{{\rm sym}}_{f^{*}}$ implies that $\mathcal{J}_{F^{*}}$ restricts to be the inverse of  $\mathcal{J}_{F}\colon\mathcal{A}_F\to\mathcal{A}_{F^{*}}$.
\end{proof}

\smallskip
\begin{exa}[Constant rank manifolds]
{\rm
Consider the closed convex cones of positive (respectively negative) semi-definite matrices ${\bf S}^n_{+}$ (respectively ${\bf S}^n_{-}$). Clearly, we have equality ${\bf S}^n_{\pm}=\lambda^{-1}(\R^n_{\pm})$.
Define the constant rank manifolds
$$M^{\pm}_k:=\{X\in{\bf S}^n_{\pm}: \rank X=k\}, \quad\textrm{ for } k=0,\ldots, n.$$ Then using Theorem~\ref{thm:liftdual} one can easily check that the manifolds $M^{\pm}_k$ and $M^{\mp}_{n-k}$ are dual to each other under the conjugacy correspondence $\delta_{{\bf S}^n_{+}}\xleftrightarrow{*}\delta_{{\bf S}^n_{-}}$.
}
\end{exa}

\section{Extensions to nonsymmetric matrices}\label{sec:nonsymm}
Consider the space of $n\times m$ real matrices ${\bf M}^{n\times m}$, endowed with the trace inner-product $\langle X,Y\rangle=\tr(X^T Y)$, and the corresponding Frobenius norm. We will let the group ${\bf O}^{n,m}:={\bf O}^n\times {\bf O}^m$ act on ${\bf M}^{n\times m}$ simply by defining
$$(U,V).X=U^T X V \textrm{ for all } (U,V)\in {\bf O}^{n,m} \textrm{ and } X\in {\bf M}^{n\times m}.$$ Recall that {\em singular values} of a matrix $A\in {\bf M}^{n\times m}$ are defined to be the square roots of the eigenvalues of the matrix $A^{T}A$.
The {\em singular value mapping} $\sigma:{\bf M}^{n\times m}\to\R^m$ is simply the mapping taking each matrix $X$ to its vector $(\sigma_1(X),\ldots,\sigma_m(X))$ of singular values in non-increasing order. We will be
interested in functions $F\colon {\bf M}^{n\times m}\to\overline{\R}$ that are invariant under the action of ${\bf O}^{n,m}$. Such functions $F$ can necessarily be represented as a composition $F=f\circ\sigma$, where the outer-function $f\colon\R^m\to\overline{\R}$ is {\em absolutely permutation-invariant}, meaning invariant under all signed permutations of coordinates. As in the symmetric case, it is useful to localize this notion. Namely, we will say that a function $f$ is {\em locally absolutely permutation-invariant} around a point $\bar{x}$ provided that for each signed permutation $\sigma$ fixing $\bar{x}$, we have $f(\sigma x)=f(x)$ for all $x$ near $\bar{x}$. Then essentially all of the results presented in the symmetric case have natural analogues in this setting (with nearly identical proofs).

\smallskip
\begin{thm}[The nonsymmetric case: lifts of manifolds]
Consider a matrix $\bar{X}\in{\bf M}^{n\times m}$ and a set $M\subset\R^m$ that
is locally absolutely permutation-invariant around $\bar{x}:=\sigma(\bar{X})$.
Then $M$ is a ${\bf C}^{\infty}$ manifold around $\bar{x}$ if and only if the
set $\sigma^{-1}(M)$ is a ${\bf C}^{\infty}$ manifold around $\bar{X}$.
\end{thm}
\smallskip

\begin{prop}[The nonsymmetric case: lifts of identifiable sets]
Consider a lsc $f\colon\R^m\to\overline{\R}$ and a matrix $\bar{X}\in {\bf M}^{n\times m}$.
Suppose that $f$ is locally absolutely permutation-invariant around $\bar{x}:=\sigma(\bar{X})$
and consider a subset $M\subset \R^m$ that is locally absolutely permutation-invariant around $\bar{x}$.
Then $M$ is identifiable (relative to $f$) at $\bar{x}$ for $\bar{v}\in\partial f(\bar{x})$, if and only if
$\sigma^{-1}(M)$ is identifiable (relative to $f\circ\sigma$) at $\bar{X}$ for $U^{T}(\Diag \bar{v}) V\in\partial (f\circ \sigma)(\bar{X})$,
where $(U,V)\in {\bf O}^{n,m}$ is any pair satisfying $\bar{X}=U^T(\Diag \sigma(\bar{X}))V$.
\end{prop}
\smallskip

\begin{thm}[The nonsymmetric case: lifts of partly smooth manifolds]
Consider a lsc function $f\colon\R^m\to\overline{\R}$ and a matrix
$X\in{\bf M}^{n\times m}$. Suppose that $f$ is locally absolutely
permutation-invariant around $\bar{x}:={\sigma(\bar{X})}$. Then $f$
is ${\bf C}^{\infty}$-partly smooth at $\bar{x}$ relative to $M$ if
and only if $f\circ\sigma$ is ${\bf C}^\infty$-partly smooth at
$\bar{X}$ relative to $\sigma^{-1}(M)$.
\end{thm}
\smallskip

It is unknown whether the analogue of the latter theorem holds in the case of ${\bf C}^p$ partial
smoothness, where $p<\infty$. This is so because it is unknown whether a nonsymmetric analogue of
\cite[Theorem 4.21]{man} holds in case of functions that are differentiable only finitely many times.

Finally, we should note that Section~\ref{sec:dual} also has a natural analogue in the nonsymmetric setting.
For the sake of brevity, we do not record it here.

\section*{Acknowledgments}
The first author thanks Nicolas Hadjisavvas
for useful discussions leading to a simplification of the proof of
Lemma~\ref{nhad}.

\bibliographystyle{plain}
\small
\parsep 0pt
\bibliography{dim_graph}

\end{document}